\documentclass[final,leqno]{siamltex}
\usepackage{amsmath,amssymb,bm}
\title{Existence and non-existence of global solutions for a discrete semilinear heat equation}
\author{Keisuke~Matsuya\footnotemark[2]  and Tetsuji~Tokihiro\footnotemark[2]}
\begin{document}
\maketitle
\pagestyle{myheadings}
\thispagestyle{plain}
\markboth{KEISUKE~MATSUYA AND TETSUJI~TOKIHIRO}{GLOBAL SOLUTIONS FOR A DISCRETE SEMILINEAR HEAT EQ}
\renewcommand{\thefootnote}{\fnsymbol{footnote}}
\footnotetext[2]{Graduate School of Mathematical Sciences, University of Tokyo, Komaba 3-8-1, Meguro, Tokyo
153-8914, Japan~(matsuya@ms.u-tokyo.ac.jp, toki@ms.u-tokyo.ac.jp).}
\renewcommand{\thefootnote}{arabic{footnote}}
\begin{abstract}
Existence of global solutions to initial value
problems for a discrete analogue of a $d$-dimensional
semilinear heat equation is investigated. We prove
that a parameter $\alpha$ in the partial
difference equation plays exactly the same role as the
parameter of nonlinearity does in the semilinear heat equation.
That is, we prove non-existence of a non-trivial global solution for $0<\alpha \le 2/d$, and, for $\alpha > 2/d$, existence of non-trivial global solutions  for sufficiently small initial data.
\end{abstract}
%
%
\section{Introduction}
The blowing up of solutions to the semilinear heat equation has been analysed extensively since the pioneering work by Fujita\cite{Fujita, Bebernes_Eberly, Levine}.
Fujita studied the initial value problem of the equation:
\begin{equation}
\begin{cases}
\displaystyle{\frac{\partial f}{\partial t} = \Delta f +f^{1+\alpha}} \ (\alpha >0)\\
f(0,\vec{x}) = a(\vec{x}) \geq 0 \ \ (a(\vec{x}) \not\equiv 0) \label{eq:1.1}
\end{cases}
,
\end{equation}
where $f:=f(t,\vec{x})\ (t \geq 0,\vec{x} \in \mathbb{R}^d)$ and $\Delta$ is the $d$-dimensional Laplacian \ 
$
\Delta := \sum\limits_{k=1}^{d}{\frac{\partial ^2}{\partial x_k^2}}.
$
When the initial value $a(\vec{x})$ is continuous and unifomly bounded, there is a smooth solution for $t>0$ and whenever the solution is bounded, the solution is prolongable.
Moreover, since $a(\vec{x})\geq 0$, the solution satisfies $f(t,\vec{x})\geq 0$.
A feature of \eqref{eq:1.1} is that the solution is not necessarily bounded for all $t\geq 0$.
This fact is easily understood if one considers spatially uniform initial condition, $a(\vec{x})\equiv a \in \mathbb{R}_+$.
In this case,  $f(t,\vec{x})=f(t)$ and \eqref{eq:1.1} becomes an ordinary differential equation, 
\begin{equation}\label{eq:1.3}
\begin{cases}
\displaystyle{\frac{d f}{d t} = f^{1+\alpha}} \\
f(0)=a > 0
\end{cases}
.
\end{equation}
The solution to this differential equation is
\begin{equation*}
\displaystyle{f(t) = \frac{\alpha^{-1/\alpha}}{(\alpha^{-1}a^{-\alpha}-t)^{1/\alpha}}},
\end{equation*}
and we see that it diverges as $t \rightarrow \alpha^{-1}a^{-\alpha}-0$.
In general, if there exists a finite time $T \in \mathbb{R}_+$ and if the solution of $(\ref{eq:1.1})$ in $(t,\vec{x})\in[0,T)\times\mathbb{R}^d$ satisfies
\begin{equation*}
\limsup\limits_{t\to T-0}{\| f(t,\cdot)\|_{L^{\infty}}}=\infty,
\end{equation*}
where
\begin{equation*}
\| f(t,\cdot)\|_{L^\infty}:=\sup\limits_{\vec{x}\in\mathbb{R}^d}{|f(t,\vec{x})|},
\end{equation*}
then we say that the solution of $(\ref{eq:1.1})$ blows up at time $T$, and therefore, that it is not a global solution (in time). 
In 1966, Fujita\cite{Fujita} proved the following theorem

\begin{theorem}\label{th1.1}
\begin{remunerate}
\item[$(1)$] If $0 < \alpha < 2/d$, any solution to $(\ref{eq:1.1})$ is not a global solution in time.
\item[$(2)$] If $2/d< \alpha$, then global solutions to $(\ref{eq:1.1})$ do exist for sufficiently small and smooth initial functions $a(\vec{x})$.
\end{remunerate}
\end{theorem}

A remarkable point in this theorem is that the parameter $\alpha$ affects the existence of the global solution of $(\ref{eq:1.1})$. 
The critical value $2/d$ in this theorem is called the Fujita exponent. 
The case $\alpha = 2/d$ was studied by Hayakawa for $d=1,2$, and by Kobayashi, Sirao, Tanaka\cite{Kobayashi_Sirao_Tanaka}, and Weissler\cite{Weissler} for general $d$.
They proved,

\begin{theorem}\label{th1.2}
For the case $\alpha = 2/d$, there exists no global solution to \eqref{eq:1.1} .
\end{theorem}

In numerical computation of \eqref{eq:1.1}, one has to discretize it and consider a partial difference equation.
A naive discretization would be to replace the $t$-differential with a forward difference and the Laplacian with a central difference such that \eqref{eq:1.1} turns into
\begin{equation*}
\displaystyle{\frac{f^{\tau+1}_{\vec{n}}-f^\tau_{\vec{n}}}{\delta} = \sum\limits_{k=1}^d{\frac{f^\tau_{\vec{n}+\vec{e}_k}-2f^\tau_{\vec{n}}+f^\tau_{\vec{n}-\vec{e}_k}}{\xi^2}}+(f^\tau_{\vec{n}})^{1+\alpha}},
\end{equation*}
where $f(\tau,\vec{n})(=:f^\tau_{\vec{n}}):\mathbb{Z}_{{} \geq 0} \times \mathbb{Z}^d \to \mathbb{R}$, for positive constants $\delta$ and  $\xi$Cand where $\vec{e}_k \in \mathbb{Z}^d$\ is the unit vector whose $k$th component is $1$ and whose other components are $0$.
Putting $\lambda := \delta/\xi^2$, we obtain
\begin{equation}\label{eq:1.2}
f^{\tau+1}_{\vec{n}} = 2d\lambda\hat{M}(f^\tau_{\vec{n}})+(1-2d\lambda)f^\tau_{\vec{n}} + \delta (f^\tau_{\vec{n}})^{1+\alpha} \quad (\alpha>0).
\end{equation}
Here
\begin{equation}\label{def:M}
\hat{M}(V_{\vec{n}}) := \displaystyle{\frac{1}{2d}}\sum\limits_{k=1}^{d}{(V_{\vec{n}+\vec{e}_k}+V_{\vec{n}-\vec{e}_k})}.
\end{equation}
For a spatially uniform initial condition, $(\ref{eq:1.2})$ becomes an ordinary difference equation
\begin{equation*}
f^{\tau+1} = f^\tau + \delta (f^\tau)^{1+\alpha}.
\end{equation*}
The above equation is a discretization of \eqref{eq:1.3}, but the features of its solutions are quite different.
In fact, $f^\tau$ will never blow up at finite time steps.
Hence, \eqref{eq:1.2} does not preserve the global nature of the original semilinear heat equation $(\ref{eq:1.1})$. 

In this article, we propose and investigate a discrete analogue of $(\ref{eq:1.1})$ which does keep the important characteristic of existence and non-existence of the global solutions in time.
In section 2, we present a partial difference equation with a parameter $\alpha$ whose continuous limit equals to  \eqref{eq:1.1}, and state the main theorem which shows that this difference equation has exactly the same properties as \eqref{eq:1.1} with respect to $\alpha$. 
This theorem is proved in section 3 for the case $0<\alpha<2/d$. In section 4 for $2/d<\alpha$ and  in section 5 for $\alpha=2/d$.
%
%
\section{Discretization of the semilinear heat equation}\label{sec2}

We consider the following initial value problem for the partial difference equation
\begin{equation}
\begin{cases}
\displaystyle{f^{\tau+1}_{\vec{n}}= \frac{g^\tau_{\vec{n}}}{\{ 1 - (g^\tau_{\vec{n}})^{\alpha} \} ^{1/\alpha}}}  \quad (\tau \in \mathbb{Z}_{{} \geq 0},\ \vec{n} \in \mathbb{Z}^d)\\
f^{0}_{\vec{n}} = a (\vec{n}) \ge 0 \quad (a (\vec{n})  \not\equiv 0)\label{eq:2.1},
\end{cases}
\end{equation}
where  $\alpha >0$ is a parameter and $g^\tau_{\vec{n}}$ is defined by means of $\hat{M}$ \eqref{def:M} as
\begin{equation*}
g^\tau_{\vec{n}} := \hat{M}(f^\tau_{\vec{n}}).
\end{equation*}
By scaling  $f$ with a positive parameter $\delta$ as
\begin{eqnarray*}
F^\tau_{\vec{n}} &:=& (\alpha \delta)^{-1/\alpha}f^\tau_{\vec{n}}, \\
G^\tau_{\vec{n}} &:=& \hat{M}(F^\tau_{\vec{n}}),
\end{eqnarray*}
we have
\begin{equation}
F^{\tau+1}_{\vec{n}}  = \frac{G^\tau_{\vec{n}}}{\{ 1 - \alpha \delta (G^\tau_{\vec{n}})^{\alpha} \} ^{1/\alpha}}\label{eq:2.2}.
\end{equation}
If there exists a smooth function $F(t, \vec{x}) $  $(t \in \mathbb{R}_{{} \geq 0},\ \vec{x} \in \mathbb{R}^d )$ that satisfies 
$F(\tau \delta, \xi \vec{n})=F^\tau_{\vec{n}} $ with $\xi:=\sqrt{2d\delta} $, we find
\[
F(t+\delta, \vec{x})= G(t,\vec{x}) \{ 1+\delta(G(t,\vec{x}))^{\alpha} \} + O(\delta^2),
\]
with
\[
G(t,\vec{x}):= \frac{1}{2d}\sum_{k=1}^d\left(F(t, \vec{x}+\xi\vec{e}_k)-2F(t, \vec{x})+F(t, \vec{x}-\xi\vec{e}_k)\right),
\]
or
\begin{align*}
\frac{F(t+\delta, \vec{x}) - F(t, \vec{x})}{\delta}&= \sum^d_{k=1}\frac{F(t, \vec{x}+\xi\vec{e}_k)-2F(t, \vec{x})+F(t, \vec{x}-\xi\vec{e}_k)}{\xi ^2}\\
&\qquad \qquad + (F(t, \vec{x}))^{1+\alpha} + O(\delta).
\end{align*}
Taking the limit $\delta \to +0$,  we obtain the semilinear heat equation \eqref{eq:1.1}
\begin{equation*}
\frac{\partial F}{\partial t} = \Delta F + F^{1+\alpha}.
\end{equation*}
Thus \eqref{eq:2.1} can be regarded as a discrete analogue of  \eqref{eq:1.1}.

Because of the term $\displaystyle \left(1-(g^\tau_{\vec{n}})\right)^{1/\alpha}$, if $g^\tau_{\vec{n}} \to 1-0$, then $f^\tau_{\vec{n}} \to +\infty$,
and \eqref{eq:2.1} cannot be defined when  $g^\tau_{\vec{n}} \ge 1$ for generic $\alpha$.
This behaviour may be regarded as an analogue of the blow up of solutions for the semilinear heat equations. 
Thus we define a global solution of \eqref{eq:2.1} as follows.
\begin{definition}
When a solution $f^\tau_{\vec{n}}$ of \eqref{eq:2.1} is non-negative and uniquely determined for all $\tau\in \mathbb{Z}_{{}_\ge0}$ and $\vec{n} \in \mathbb{Z}^d$, i.e. $\ g^\tau_{\vec{n}}<1\ $ for all $\tau\in \mathbb{Z}_{{}_\ge0}$ and $\vec{n} \in \mathbb{Z}^d$, then we say that the solution  $f^\tau_{\vec{n}}$ is a global solution (in time) of \eqref{eq:2.1}.
\end{definition}

The advantage of using $(\ref{eq:2.1})$ instead of \eqref{eq:1.2} is apparant from the following proposition.
\begin{proposition}\label{prop2.4}
No spatially uniform function can be a global solution of $(\ref{eq:2.1})$.
\end{proposition}
\begin{proof}
When the solution is spatially uniform, then $(\ref{eq:2.1})$ takes the form of an ordinary difference equation
\begin{eqnarray*}
f^{\tau +1} &=& \frac{f^\tau}{\{ 1-(f^\tau)^{\alpha} \} ^{1/\alpha}}\\
 &=& \{ (f^\tau)^{-\alpha} -1 \} ^{-1/\alpha}\quad (\mbox{for $f \ne 0$}).
 \end{eqnarray*}
 Hence, if a non-trivial global solution exists, it satisfies $0 < f^\tau <1$ for all $\tau \in \mathbb{Z}_{{}_\ge 0}$ and 
 \begin{eqnarray*}
(f^{\tau +1})^{-\alpha} &=& (f^\tau)^{-\alpha}-1.
\end{eqnarray*}
However, the above equation is easily solved and we get
\begin{equation*}
(f^\tau)^{-\alpha} = (f^0)^{-\alpha} - \tau
\end{equation*}
Since $(f^0)^{-\alpha}$ is a positive constant,  $f^\tau$ cannot be defined for $\tau \ge (f^0)^{-\alpha}$, which is a contradiction. 
\end{proof}

%
%
Furthermore, \eqref{eq:2.1} inherits quite similar properties to those of \eqref{eq:1.1}. The following theorem is the main result in this article.
\begin{theorem}\label{th2.1}
\begin{remunerate}
\item[$(1)$] For $0 < \alpha < 2/d$, there is no global solution to $(\ref{eq:2.1})$.
\item[$(2)$] For  $\alpha = 2/d$, there is no global solution to $(\ref{eq:2.1})$.
\item[$(3)$] Let $\| f^0\|_1 $ be the $l^1$ norm of the initial function, i.e., $\; \| f^0\|_1 :=\sum_{\vec{n}} |f^0_{\vec{n}}|$.
 For $2/d < \alpha$, if $\; \|f^0\|_1\;$ is sufficiently small, then global solutions to $(\ref{eq:2.1})$ exist.
\end{remunerate}
\end{theorem}

%
{\em Remark.}\
$\bullet$ In the limit $\alpha \to +0$ in $(\ref{eq:2.2})$, we have
\begin{equation*}
\begin{cases}
F^{\tau +1}_{\vec{n}} &= e^\delta G^\tau_{\vec{n}}\\
F^0_{\vec{n}}&=a(\vec{n}) \ge 0 \quad (a(\vec{n}) \not\equiv 0)
\end{cases}
\end{equation*}
Although $\|F^\tau\|_1 $ diverges as $\tau \to + \infty$ as far as $\delta >0$, the solution of this difference equation is a global solution in time.
On the other hand, when $\alpha = 0$, the partial differential equation \eqref{eq:1.1} becomes linear, and its solution is also a global solution.
\\
$\bullet$ To consider negative solutions or oscillatory solutions, we have only to use a slightly modified partial difference equation
\[
\displaystyle{f^{\tau+1}_{\vec{n}}= \frac{g^\tau_{\vec{n}}}{\{ 1 - (\left|g^\tau_{\vec{n}}\right|)^{\alpha} \} ^{1/\alpha}}}.
\]

%
%
\section{Proof of theorem \ref{th2.1} for $\bm{ 0<\alpha<2/d}$}\label{sec3}
The idea of the proof in this and the following sections is similar to that  adopted by Meier in the case of partial differential equations \cite{Meier}. 
We construct a subsolution and a supersolution of $f^\tau_{\vec{n}}$ and prove the existence and non-existence of the global solutions.

We denote by $U^\tau_{\vec{n}}$ the solution to the initial value problem of the linear partial difference equation
\begin{equation}
\begin{cases}
U^{\tau+1}_{\vec{n}} &= \hat{M}(U^\tau_{\vec{n}}) \\
U^0_{\vec{n}} &= \delta_{0,\vec{n}}\label{eq:4.1}
\end{cases}.
\end{equation}
Using $U^\tau_{\vec{n}}$, we define 
\begin{equation}\label{2:hfunction}
h^\tau_{\vec{n}} := \sum\limits_{\vec{n}^\prime}{U^\tau_{\vec{n}-\vec{n}^\prime}f^0_{\vec{n}^\prime}},
\end{equation}
and 
\begin{equation}\label{2:subfunction}
\displaystyle{\underline{f}^\tau_{\vec{n}} :=\frac{h^\tau_{\vec{n}}}{\{ 1- \tau(h^\tau_{\vec{n}})^\alpha \} ^{1/\alpha}}},
\end{equation}
provided that $1-\tau(h_{\vec{n}}^\tau)^\alpha>0$.

%
%
{\em Remark.} \
$\bullet$ Since the support of $U^\tau$ is finite,  the summation in the definition of $h^\tau_{\vec{n}}$ is over a finite number of lattice points $\vec{n}$.\\
$\bullet$ Due to the definition of $U_{\vec{n}}^\tau$, it holds that $\displaystyle \left\|  h^\tau \right\|_1 = \left\| f^0 \right\|_1 $.\\
$\bullet$ The function  $h^\tau_{\vec{n}}$ satisfies the linear partial difference equation
\begin{eqnarray}\label{2:eq.h}
\begin{cases}
h^{\tau+1}_{\vec{n}} = \hat{M}(h^\tau_{\vec{n}}) \\
h^0_{\vec{n}} = f^0_{\vec{n}}
\end{cases}.
\end{eqnarray}
\noindent
$\bullet$ Since $h^\tau_{\vec{n}} \ge 0$, $\underline{f}^\tau_{\vec{n}} \ge 0$ if it exists.\\
$\bullet$ When $h^\tau_{\vec{n}} > 0$ and $1-\tau (h_{\vec{n}}^{\tau})^\alpha>0$,
\begin{equation}\label{3:hf}
\left( \underline{f}^\tau_{\vec{n}} \right)^{-\alpha}=\left( h^\tau_{\vec{n}} \right)^{-\alpha}-\tau.
\end{equation}
\noindent
$\bullet$ The function  $\underline{f}^\tau_{\vec{n}}$ does not always exist for all $\tau \in \mathbb{Z}_{{}_\ge 0}$.
When $1-\tau(h_{\vec{n}}^\tau)^\alpha \le 0$, $\, \underline{f}^\tau_{\vec{n}} $ does not exist.

To  understand the meaning of $\underline{f}^\tau_{\vec{n}}$ and in order to give the proof of theorem \ref{th2.1} (1), we need several lemmas and propositions.
First we consider the following elementary inequality.
%
%
\begin{lemma}\label{lem4.2}
Let $x_k$ ($k=1,2,...,N)$ be non-negative real numbers, $\alpha>0$,  and $\displaystyle \sum_{k=1}^N x_k > 0$.
It then holds that 
\begin{equation*}
\displaystyle{\frac{N^\alpha\sum\limits^N_{k=1}{(x_k)^{1+\alpha}}}{\left( \sum\limits^N_{k=1}{x_k} \right)^{1+\alpha}}\geq 1}.
\end{equation*}
\end{lemma}
%
%
\begin{proof}
If we put
\begin{equation*}
\phi(x_1,\cdots,x_N) := N^\alpha\sum\limits^N_{k=1}{(x_k)^{1+\alpha}} - \left( \sum\limits^N_{k=1}{x_k}\right)^{1+\alpha},
\end{equation*}
lemma \ref{lem4.2} is equivalent to the inequality $\phi(x_1,\cdots,x_N) \ge 0$.
Without loss of generality, we can assume
$
\; x_1 \ge x_2 \ge \cdots \ge x_N \ge 0.
$
From  the partial differentiation
\begin{align*}
&\frac{\partial}{\partial x_{k+1}} \phi(x_1,\cdots,x_{k}, \underbrace{x_{k+1},x_{k+1},\cdots,x_{k+1}}_{N-k}) \qquad  (k=1,2,\ldots,N-1) \\
&=(N-k)(1+\alpha ) \left\{ (Nx_{k+1})^\alpha - \left( \sum_{j=1}^{k}{x_j} + (N-k)x_{k+1} \right)^\alpha \right\} \le 0,
\end{align*}
we find
\[
\phi(x_1,\cdots,x_{k},\underbrace{x_{k+1}, x_{k+1},\cdots,x_{k+1}}_{N-k}) \ge \phi(x_1,\cdots,x_{k-1},\underbrace{x_k,x_k,\cdots,x_k}_{N-k+1}).
\]
Hence we have, recursively,
\begin{eqnarray*}
\phi(x_1,\cdots,x_{N-1},x_{N}) &\geq& \phi(x_1,\cdots,x_{N-2},x_{N-1},x_{N-1}) \\
&\geq& \phi(x_1,\cdots,x_{N-3}, x_{N-2},x_{N-2},x_{N-2}) \\
&\geq& \cdots \geq \phi(x_1, \cdots ,x_1) = 0. 
\end{eqnarray*}
\end{proof}

%
%
\begin{lemma}\label{lem4.1}
Suppose that $\underline{f}^\tau_{\vec{n}}$  and $\underline{f}^\tau_{\vec{n}\pm\vec{e_k}}$ $(k=1,2,\ldots,d)$ exist and that $\hat{M}( \underline{f}^\tau_{\vec{n}})>0$.
If $\{\hat{M}(\underline{f}^\tau_{\vec{n}})\}^{-\alpha} -1 >0$, then
$\underline{f}^{\tau+1}_{\vec{n}}$ exists and satisfies
\begin{equation}
(\underline{f}^{\tau+1}_{\vec{n}})^{-\alpha} \geq \{\hat{M}(\underline{f}^\tau_{\vec{n}})\}^{-\alpha} -1\label{eq:4.2}.
\end{equation}
\end{lemma}
%
%
{\em Remark}\  $\bullet$ The condition $ \hat{M}(\underline{f}^\tau_{\vec{n}})>0$ implies that at least one $\underline{f}^\tau_{\vec{n}\pm \vec{e}_k}$ is positive and  that $h^{\tau+1}_{\vec{n}}=\hat{M}(h^\tau_{\vec{n}})>0$.\\
$\bullet$ When $\hat{M}(h^\tau_{\vec{n}})=0$, then $h^{\tau+1}_{\vec{n}}=0$ and $ \underline{f}^{\tau+1}_{\vec{n}}=0$.

\vspace*{2mm}
%
%
\begin{proof}
We assume that there are $m$ ($1\le m \le 2d$) nonzero values among $\displaystyle \{ h^\tau_{\vec{n}\pm\vec{e}_k}\}_{k=1}^d$, and denote them
by $h_1, h_2, \ldots, h_m$. Accordingly we put 
\[
f_i:=\frac{h_i}{\left(1-\tau (h_i)^\alpha\right)^{1/\alpha}}=(h_i^{-\alpha}-\tau)^{-1/\alpha} \qquad (i=1,2,\ldots,m).
\]
By definition,
\begin{align*}
\chi(\tau)&:=\{\hat{M}(h^\tau_{\vec{n}})\}^{-\alpha} - \tau - \{ \hat{M}(\underline{f}^\tau_{\vec{n}})\}^{-\alpha} \\
&= \left\{\displaystyle{\frac{1}{2d}\sum\limits^d_{k=1}{(h^\tau_{\vec{n}+\vec{e}_k}+h^\tau_{\vec{n}-\vec{e}_k})}}\right\}^{-\alpha} - \tau 
- \left\{\displaystyle{\frac{1}{2d}\sum\limits^d_{k=1}(\underline{f}^\tau_{\vec{n}+\vec{e}_k}+\underline{f}^\tau_{\vec{n}-\vec{e}_k})}\right\}^{-\alpha}\\
&=\left(\frac{1}{2d}\sum_{i=1}^m h_i\right)^{-\alpha} - \tau -\left(\frac{1}{2d}\sum_{i=1}^m f_i\right)^{-\alpha}\\
&=\left(\frac{1}{2d}\sum_{i=1}^m h_i\right)^{-\alpha} - \tau -\left(\frac{1}{2d}\sum_{i=1}^m \{(h_i)^{-\alpha}-\tau\}^{-1/\alpha}\right)^{-\alpha}. 
\end{align*}
Regarding $\tau$ as a continuous variable, 
\begin{align*}
\frac{d \chi(\tau)}{d \tau}&=-1+(2d)^{\alpha}\frac{\sum_{i=1}^m \{(h_i)^{-\alpha}-\tau\}^{-(\alpha+1)/\alpha}}{\left[  \sum_{i=1}^m\{ (h_i)^{-\alpha}-\tau  \}^{-1/\alpha} \right]^{\alpha+1}}\\
&\ge-1+(m)^{\alpha}\frac{\sum_{i=1}^m \{(h_i)^{-\alpha}-\tau\}^{-(\alpha+1)/\alpha}}{\left[  \sum_{i=1}^m\{ (h_i)^{-\alpha}-\tau  \}^{-1/\alpha} \right]^{\alpha+1}}.
\end{align*}
Appling lemma \ref{lem4.2} for
$
N=m, \ x_i= \{ (h_i)^{-\alpha}-\tau  \}^{-1/\alpha},
$
we find  $\displaystyle \frac{d \chi(\tau)}{d \tau} \ge 0$.
Thus $\chi(\tau) \ge \chi(0)=0$, and we obtain
\begin{equation*}
\left\{\hat{M}(h^\tau_{\vec{n}})\right\}^{-\alpha}-\tau\geq\left\{\hat{M}(\underline{f}^\tau_{\vec{n}})\right\}^{-\alpha}.
\end{equation*}
By the assumption $\{\hat{M}(\underline{f}^\tau_{\vec{n}})\}^{-\alpha} -1 >0$ and \eqref{2:subfunction}  we have
\[
(h_{\vec{n}}^{\tau+1})^{-\alpha}-(\tau+1)=\{\hat{M}(h^\tau_{\vec{n}})\}^{-\alpha}-\tau-1\geq\{\hat{M}(\underline{f}^\tau_{\vec{n}})\}^{-\alpha}-1>0.
\]
Thus we find  
\[
1-(\tau+1)(h_{\vec{n}}^{\tau+1})^{\alpha}>0.
\]
Therefore $\underline{f}^{\tau+1}_{\vec{n}}$ exists and due to \eqref{3:hf} it satisfies
\[
(\underline{f}_{\vec{n}}^{\tau+1})^{-\alpha}\geq\{\hat{M}(\underline{f}^\tau_{\vec{n}})\}^{-\alpha}-1.
\]
\end{proof}

%
%
%
\begin{proposition}\label{prop4.1}
If the solution $f^\tau_{\vec{n}}$ of the initial value problem \eqref{eq:2.1} exists at $\tau$ and for all $\vec{n} \in \mathbb{Z}^d$, 
then $\underline{f}^\tau_{\vec{n}}$ exists at the same $\tau$ and for all $\vec{n} \in \mathbb{Z}^d$, and satisfies
\begin{equation}\label{2:eqprop4.1}
\underline{f}^\tau_{\vec{n}} \leq f^\tau_{\vec{n}}.
\end{equation}
\end{proposition}
%
%
\begin{proof} \quad 
The proof goes by induction for $\tau$.
Since $\ \displaystyle \underline{f}^0_{\vec{n}} = f^0_{\vec{n}}$ by definition,  $\underline{f}^\tau_{\vec{n}}$ exists and satisfies \eqref{2:eqprop4.1} for
$\tau=0$ and for all $\vec{n}$.
Suppose that $\underline{f}^\tau_{\vec{n}}$ exists and satisfies  \eqref{2:eqprop4.1} for all $\vec{n}$ and for all $\tau \leq s$.
If $f^{s+1}_{\vec{n}}$ exists, either $f^{s+1}_{\vec{n}}=0$ or $f^{s+1}_{\vec{n}}>0$.
When $f^{s+1}_{\vec{n}}=0$, it implies 
\begin{eqnarray*}
f^{s+1}_{\vec{n}}=0 & \Longleftrightarrow &\hat{M}(f^{s}_{\vec{n}})=0 \\
& \Longleftrightarrow &f^{s}_{\vec{n}\pm \vec{e}_k}=0 \quad (k=1,2,\ldots,d)\\
& \Longrightarrow &\underline{f}^{s}_{\vec{n}\pm \vec{e}_k}=0 \quad (k=1,2,\ldots,d)\\
& \Longleftrightarrow &h^{s}_{\vec{n}\pm \vec{e}_k}=0 \quad (k=1,2,\ldots,d)\\
& \Longleftrightarrow &\hat{M}(h^{s}_{\vec{n}})=0 \\
& \Longleftrightarrow &h^{s+1}_{\vec{n}}=0 \\
& \Longleftrightarrow &\underline{f}^{s+1}_{\vec{n}}=0.
\end{eqnarray*}
Hence $\underline{f}^{s+1}_{\vec{n}}$ exists and satisfies $\underline{f}^{s+1}_{\vec{n}} \le f^{s+1}_{\vec{n}}$.

When $f^{s+1}_{\vec{n}}>0$, if $ \hat{M}(\underline{f}^{s}_{\vec{n}} )=0$, then $ \hat{M}(h^{s}_{\vec{n}} )=0$ and $\underline{f}^{s+1}_{\vec{n}}=0$.
Otherwise
\begin{equation*}
\{\hat{M}(\underline{f}^{s}_{\vec{n}})\}^{-\alpha}-1 \geq \{\hat{M}(f^s_{\vec{n}})\}^{-\alpha}-1 = (f^{s+1}_{\vec{n}})^{-\alpha} > 0.
\end{equation*}
Then, from lemma \ref{lem4.1}, the existence of $\underline{f}^{s+1}_{\vec{n}}$ follows.
Moreover, from \eqref{eq:4.2},
\begin{eqnarray*}
(\underline{f}^{s+1}_{\vec{n}})^{-\alpha} &\geq& \{\hat{M}(\underline{f}^s_{\vec{n}})\}^{-\alpha} -1 \\
&\geq& \{\hat{M}(f^s_{\vec{n}})\}^{-\alpha} -1 = (f^{s+1}_{\vec{n}})^{-\alpha}.
\end{eqnarray*}
Thus we find
\begin{equation*}
\underline{f}^{s+1}_{\vec{n}} \leq f^{s+1}_{\vec{n}}.
\end{equation*}
From the induction hypothesis, \eqref{2:eqprop4.1} holds for arbitrary $\tau \in \mathbb{Z}_+$. 
\end{proof}

%
%
The following asymptotic evaluation of $U^\tau_{\vec{n}}$ is well known from the analysis of the transition probability of $d$-dimensional simple random walk.
See for example F.~Spitzer$\cite{F.Spitzer}$.
\begin{proposition}\label{prop3.2}
\begin{equation}\label{3:asymU}
\displaystyle{U^\tau_{\vec{n}} \sim 2\left( \frac{d}{4\pi}\right)^{d/2}\tau^{-d/2}} \ (\tau\to+\infty).
\end{equation}
Here $A^\tau\sim B^\tau \ (\tau\to+\infty)$ means $\lim_{\tau\to+\infty}{(A^\tau/B^\tau)}=1$.
\end{proposition}

%
%
Now, we are ready to give the proof of theorem \ref{th2.1}  (1). 

{\em Proof of theorem \ref{th2.1}  (1).} \ 

Suppose that the global solution $f^\tau_{\vec{n}}$ exists. 
From proposition \ref{prop3.2}, $h^\tau_{\vec{n}}$ is evaluated as
\begin{eqnarray}
h^\tau_{\vec{n}} &=& \displaystyle{\sum\limits_{\vec{n}^\prime}{U^\tau_{\vec{n}-\vec{n}^\prime}f^0_{\vec{n}}}} \nonumber\\
&\sim& \displaystyle{\frac{1}{\sqrt{\tau^d}}\sum\limits_{\vec{n}}{2\left(\frac{d}{4\pi}\right)^{d/2}f^0_{\vec{n}}} \quad (\tau\to+\infty)}\label{eq:4.6}.
\end{eqnarray}
Putting \;
$\displaystyle C:=2\left(d/4\pi\right)^{d/2}\| f^0 \|_1$, 
\begin{equation}\label{3:eqProof}
\displaystyle{\underline{f}^\tau_{\vec{n}} \sim \frac{h^\tau_{\vec{n}}}{(1-C^\alpha\tau^{1-d\alpha/2})^{1/\alpha}}}.
\end{equation}
For $0<\alpha<2/d$, the exponent of $\tau$ in the above equation satisfies $1-d\alpha/2>0$.
However, since $C$ is a positive constant independent of $\tau$, 
$1-C^\alpha\tau^{1-d\alpha/2}<0$ for sufficiently large $\tau$ and the evaluation \eqref{3:eqProof} does not make sense.
In other words,  $\underline{f}^\tau_{\vec{n}}$ cannot exist for sufficiently large $\tau$, which contradicts
proposition \ref{prop4.1} and the statement (1) of theorem \ref{th2.1}  is therefore proved.
$\qquad\endproof$

%
%
\section{Proof of the theorem \ref{th2.1} for $\bm{ 2/d<\alpha}$}\label{sec4}
\quad First, we define a supersolution of \eqref{eq:2.1}
\begin{equation}\label{4:supersolution}
\displaystyle{\bar{f}^\tau_{\vec{n}} := \frac{h^\tau_{\vec{n}}}{\left\{1-\sum\limits^\tau_{k=0}{(m_k)^\alpha}\right\}^{1/\alpha}}},
\end{equation}
where $m_\tau$ is defined in terms of \eqref{2:hfunction} as
\begin{equation}\label{4:mtau}
m_\tau := \sup_{\vec{n}}{h^\tau_{\vec{n}}}.
\end{equation}
Of course $\bar{f}^\tau_{\vec{n}}$ is well defined only when $\displaystyle 1-\sum\limits^\tau_{k=0}(m_k)^\alpha>0 $.

%
%
\begin{proposition}\label{prop5.1}
When $\bar{f}^\tau_{\vec{n}}$ exists at $\tau$ and for all $\vec{n} \in \mathbb{Z}^d$, 
$f^\tau_{\vec{n}}$ exists and satisfies
\begin{equation}\label{5:prepineq}
\bar{f}^\tau_{\vec{n}} \geq f^\tau_{\vec{n}}.
\end{equation}
\end{proposition}

%
%
\begin{proof}
We give the proof by induction on $\tau$.
When $\tau=0$, by definition of the initial value problem,  $f^0_{\vec{n}}$ exists and \eqref{5:prepineq} holds because
\begin{equation*}
\displaystyle{\bar{f}^0_{\vec{n}} = \frac{h^0_{\vec{n}}}{\{1-(m_0)^\alpha\}^{1/\alpha}} \geq h^0_{\vec{n}} = f^0_{\vec{n}}}.
\end{equation*}
Suppose that the statement is true up to $\tau=s$ and $\bar{f}^{s+1}_{\vec{n}}$ exists.
When $\bar{f}^{s+1}_{\vec{n}}=0$, 
\begin{eqnarray*}
\bar{f}^{s+1}_{\vec{n}}=0 &\  \Longleftrightarrow \ &h_{\vec{n}}^{s+1}=0\\
&\  \Longleftrightarrow \ &\hat{M}(h_{\vec{n}}^{s})=0\\
&\  \Longleftrightarrow \ &h_{\vec{n}\pm \vec{e}_k}^{s}=0 \quad (k=1,2,\ldots, d)\\
&\  \Longleftrightarrow \ &\bar{f}_{\vec{n}\pm \vec{e}_k}^{s}=0 \quad (k=1,2,\ldots, d)\\
&\  \Longrightarrow \ &f_{\vec{n}\pm \vec{e}_k}^{s}=0 \quad (k=1,2,\ldots, d)\\
&\  \Longleftrightarrow \ &g_{\vec{n}}^{s}=0 \\
&\  \Longleftrightarrow \ &f_{\vec{n}}^{s+1}=0.
\end{eqnarray*}
Hence  \eqref{5:prepineq} holds.

When $\bar{f}^{s+1}_{\vec{n}}>0$, if $g_{\vec{n}}^s = 0 $, then $f_{\vec{n}}^{s+1}=0$ and the statement is true.
Otherwise
\begin{eqnarray*}
0<(\bar{f}^{s+1}_{\vec{n}})^{-\alpha} &=& \frac{1-\sum\limits^{s+1}_{k=0}{(m_k)^\alpha}}{(h^{s+1}_{\vec{n}})^\alpha} = \frac{1-\sum\limits^s_{k=0}{(m_k)^\alpha}}{(h^{s+1}_{\vec{n}})^\alpha} - \left(\frac{m_{s+1}}{h^{s+1}_{\vec{n}}}\right)^\alpha \\
&\leq& \frac{1-\sum\limits^s_{k=0}{(m_k)^\alpha}}{\left\{\hat{M}(h^s_{\vec{n}})\right\}^\alpha} - 1 =\frac{1}{\left\{\hat{M}(\bar{f}^s_{\vec{n}})\right\}^\alpha} - 1 \\
&\leq& (g^s_{\vec{n}})^{-\alpha} - 1.
\end{eqnarray*}
From \eqref{eq:2.1}, $(g^s_{\vec{n}})^{-\alpha} - 1=(f^{s+1}_{\vec{n}})^{-\alpha}$ and we find 
 $(\bar{f}^{s+1}_{\vec{n}})^{-\alpha} \leq (f^{s+1}_{\vec{n}})^{-\alpha}$,  i.e. $f^{s+1}_{\vec{n}} \leq \bar{f}^{s+1}_{\vec{n}}$.
Thus, from the induction hypothesis, the statement  is true for any non-negative integer $\tau$.
\end{proof}

%
%
%
Now we prove the statement$(3)$ of theorem \ref{th2.1}.

{\em Proof of theorem \ref{th2.1}  (3).} \ 

From \eqref{eq:4.6}, we obtain the asymptotic behaviour of $m_\tau$ as
\begin{equation*}
m_\tau \sim \frac{C}{\sqrt{\tau^d}} \qquad (\tau\to+\infty).
\end{equation*}
Hence with a fixed $\tau_0\in\mathbb{Z}_{{}\geq 0}$,  we can evaluate $\bar{f}^\tau_{\vec{n}}$ as
\begin{equation*}
\displaystyle{\bar{f}^\tau_{\vec{n}} \sim \frac{h^\tau_{\vec{n}}}{\left(1-\sum\limits^{\tau_0}_{k=0}{m_k} - \sum\limits^\tau_{k=\tau_0+1}\displaystyle{\frac{C^\alpha}{k^{d\alpha/2}}}\right)^{1/\alpha}} \qquad (\tau\to+\infty)}.
\end{equation*}
Since $d\alpha/2>1$,  we have $\displaystyle \sum_{k=\tau_0+1}^\infty \frac{1}{k^{d \alpha/2}} <+\infty   $.
Noticing the fact $\left\|  h^\tau \right\|_1 = \left\| f^0 \right\|_1$,  the term $\sum\limits^{\tau_0}_{k=0}{m_k}$ and the constant $C$ can be as small as possible
by choosing a small value of $\left\| f^0\right\|_1$.
Thus $\bar{f}^\tau_{\vec{n}}$ exists at any time step $\tau$ and any lattice point $\vec{n}$, if $\left\| f^0 \right\|_1$ is sufficiently small.
Then, from proposition $\ref{prop5.1}$, we find that $f^\tau_{\vec{n}}$ is a global solution of \eqref{eq:2.1}.
$\qquad\endproof$
%
%
%
%
%
%
%
\section{ Proof of theorem \ref{th2.1} for $\bm{ \alpha=2/d}$}\label{sec5}
%
In this section, we prove the statement (2) of \ref{th2.1}.
The idea of the proof is similar to that adopted by Weissler$\cite{Weissler}$.

%
%
We define the discrete Green function $G$ and the function $\tilde{f}$ as follows.
\begin{equation}\label{5:Gfunction}
G^\tau_{\vec{n}} := \left\{
\begin{array}{cl}
0 & \mbox{$(\tau = 0)$} \\
U^{\tau-1}_{\vec{n}} & \mbox{$(\tau\geq 1)$}
\end{array}
\right. ,
\end{equation}
\begin{equation}\label{5:tildeffunction}
\tilde{f}^\tau_{\vec{n}} := 
\left\{
\begin{array}{cl}
0 &\; \mbox{$(\tau = 0)$} \\
\sum\limits^{\tau-1}_{s=0}\sum\limits_{\vec{n}^\prime}{G^{\tau-s}_{\vec{n}-\vec{n}^\prime}H(g^s_{\vec{n}^\prime})}  &\; \mbox{$(\tau\geq 1)$}
\end{array}
\right.,
\end{equation}
where 
\begin{equation}
H(g):=\frac{g}{(1-g^{\alpha})^{1/\alpha}}-g.
\label{5:gdef}
\end{equation}

Since
\[
f^{\tau+1}_{\vec{n}} - g^\tau_{\vec{n}} 
= \frac{g^\tau_{\vec{n}}}{\{ 1-(g^\tau_{\vec{n}})^\alpha\}^{1/\alpha}} - g^\tau_{\vec{n}} = H(g^\tau_{\vec{n}}),
\]
the initial value problem \eqref{eq:2.1} can be rewritten as
\begin{equation}
\begin{cases}
(\hat{T}-\hat{M})f^\tau_{\vec{n}} = H(g^\tau_{\vec{n}}) \\
f^{0}_{\vec{n}} = a (\vec{n}) \ge 0 \quad (a (\vec{n})  \not\equiv 0)\label{eq:6.2},
\end{cases}
\end{equation}
where $\hat{T}$ denotes the time shift: $\hat{T} y^\tau:=y^{\tau+1}$.

For $\tau \ge 1$, 
\[
(\hat{T}-\hat{M})G^{\tau}_{\vec{n}} = G^{\tau+1}_{\vec{n}} - \hat{M}(G^\tau_{\vec{n}}) 
= U^\tau_{\vec{n}} - \hat{M}(U^{\tau-1}_{\vec{n}}) = 0,
\]
and for $\tau=0$, 
\[
(\hat{T}-\hat{M})G^0_{\vec{n}} = G^1_{\vec{n}} - \hat{M}(G^0_{\vec{n}}) 
= U^0_{\vec{n}} = \delta_{\vec{n},0}.
\]
Thus we find
\begin{equation}
(\hat{T}-\hat{M})G^\tau_{\vec{n}} = \delta_{\tau,0}\delta_{\vec{n},0}. \label{eq:6.3}
\end{equation}

We also find
\begin{equation}
(\hat{T}-\hat{M})\tilde{f}^\tau_{\vec{n}} = H(g^\tau_{\vec{n}}) \label{eq:6.4}.
\end{equation}
Because
\begin{eqnarray*}
(\hat{T}-\hat{M})\tilde{f}^\tau_{\vec{n}} &=& \tilde{f}^{\tau+1}_{\vec{n}} - \hat{M}(\tilde{f}^\tau_{\vec{n}}) \\
&=& \sum\limits^\tau_{s=0}\sum\limits_{\vec{n}^\prime}{G^{\tau+1-s}_{\vec{n}-\vec{n}^\prime}H(g^s_{\vec{n}^\prime})} + \sum\limits^{\tau-1}_{s=0}\sum\limits_{\vec{n}^\prime}{\hat{M}(G^{\tau-s}_{\vec{n}-\vec{n}^\prime}H(g^s_{\vec{n}^\prime}))} \\
&=& \sum\limits_{\vec{n}^\prime}{G^1_{\vec{n}-\vec{n}^\prime}H(g^\tau_{\vec{n}^\prime})} + \sum\limits^{\tau-1}_{s=0}\sum\limits_{\vec{n}^\prime}{(\hat{T}-\hat{M})\{G^{\tau-s}_{\vec{n}-\vec{n}^\prime} H(g^s_{\vec{n}^\prime})\}} \\
&=& \sum\limits_{\vec{n}^\prime}{U^0_{\vec{n}-\vec{n}^\prime}H(g^\tau_{\vec{n}^\prime})}+\sum\limits^{\tau-1}_{s=0}\sum\limits_{\vec{n}^\prime}{\{\delta_{\tau-s,0}\delta_{\vec{n}-\vec{n}^\prime,0}H(g^s_{\vec{n}^\prime})\}} \\
&=& H(g^\tau_{\vec{n}}).
\end{eqnarray*}




%
%
Using the above results, we can prove the following proposition.
\begin{proposition}\label{prop6.1}
When $f_{\vec{n}}^\tau$ is a global solution of \eqref{eq:2.1}
\begin{equation}\label{eq:6.1}
f^\tau_{\vec{n}} =h^\tau_{\vec{n}} + \sum\limits_{s=1}^\tau\sum\limits_{\vec{n}^\prime}{U^{\tau-s}_{\vec{n}-\vec{n}^\prime}H(g^{s-1}_{\vec{n}^\prime})}.
\end{equation}
Here $h^\tau_{\vec{n}} $ is defined in \eqref{2:hfunction}.
\end{proposition}
%
%
%
\begin{proof}
First we note that $\displaystyle  \tilde{f}^\tau_{\vec{n}}=\sum\limits_{s=1}^\tau\sum\limits_{\vec{n}^\prime}{U^{\tau-s}_{\vec{n}-\vec{n}^\prime}H(g^{s-1}_{\vec{n}^\prime})} $, and the right hand side of \eqref{eq:6.1} $= h^\tau_{\vec{n}} + \tilde{f}^\tau_{\vec{n}}$.
By definition  $h^0_{\vec{n}} = f^0_{\vec{n}}$ and $ \tilde{f}^0_{\vec{n}}=0$.
From \eqref{2:eq.h} and \eqref{eq:6.4},
\begin{equation*}
\begin{cases}
(\hat{T}-\hat{M})(h^\tau_{\vec{n}}+\tilde{f}^\tau_{\vec{n}}) = H(g^\tau_{\vec{n}}) \\
h^0_{\vec{n}}+\tilde{f}^0_{\vec{n}} = f^0_{\vec{n}}.
\end{cases}
\end{equation*}
Thus, from \eqref{eq:6.2},  $h^\tau_{\vec{n}}+\tilde{f}^\tau_{\vec{n}}$  satisfies the same initial value problem as that for $f^\tau_{\vec{n}}$. 
Since the solution of the initial value problem \eqref{eq:6.2} (\eqref{eq:2.1}) is unique, $f^\tau_{\vec{n}} =h^\tau_{\vec{n}}+ \tilde{f}^\tau_{\vec{n}}$. 
\end{proof}

%
%
%
\begin{proposition}\label{prop6.2}
Suppose that $\alpha = 2/d$ and that $f^\tau_{\vec{n}} $ is a global solution of $(\ref{eq:2.1})$. 
Then there exists a constant $C_0$ and 
\begin{equation*}
\| f^\tau\|_1 \leq C_0.
\end{equation*}
for all $\tau \in \mathbb{Z}_+$.
\end{proposition}
%
%
\begin{proof}
From proposition \ref{prop4.1}  and the assumption of the proposition,  $\underline{f}^\tau_{\vec{n}}$ also exists for all $(\tau, \vec{n})$ and 
\begin{equation}
1-(\tau^{d/2}h^\tau_{\vec{n}})^\alpha > 0\label{eq:6.6}.
\end{equation}
On the other hand, from proposition \ref{prop3.2},
\begin{equation*}
\tau^{d/2}h^\tau_{\vec{n}} \ \sim \ \sum\limits_{\vec{n}}{2\left(\frac{d}{4\pi}\right)^{d/2}f^0_{\vec{n}}} \, = \, 2\left(\frac{d}{4\pi}\right)^{d/2}\| f^0 \|_1 \qquad (\tau\to+\infty).
\end{equation*}
Then from $(\ref{eq:6.6})$,
\begin{equation*}
\left\{2\left(\frac{d}{4\pi}\right)^{d/2}\| f^0 \|_1 \right\}^\alpha < 1.
\end{equation*}
Therefore,
\begin{equation*}
\| f^0 \|_1 < \frac{1}{2}\left(\frac{4\pi}{d}\right)^{d/2}.
\end{equation*}
Because $f^\tau_{\vec{n}}$ is a global solution, if we take $f^\tau_{\vec{n}}$ at any time step $\tau$ as the initial value of $(\ref{eq:2.1})$, 
the solution is also a global solution. 
Therefore we obtain
\begin{equation*}
\| f^\tau \|_1 < \frac{1}{2}\left(\frac{4\pi}{d}\right)^{d/2}.
\end{equation*}
\end{proof}

%
%
\begin{lemma}\label{lem6.3}
When $f^\tau_{\vec{n}}$ is a global solution,
\begin{equation}\label{5:ineqhf}
h^\tau_{\vec{n}} \leq f^\tau_{\vec{n}}.
\end{equation}
\end{lemma}
%
%
\begin{proof}
We prove by induction for $\tau$.
\begin{equation*}
h^0_{\vec{n}} = \sum\limits_{\vec{n}^\prime}{U^0_{\vec{n}-\vec{n}^\prime}f^0_{\vec{n}^\prime}} = f^0_{\vec{n}}
\end{equation*}
and \eqref{5:ineqhf} holds for $\tau=0$.
Suppose that the statement is true for $\tau = s$.
then
\begin{eqnarray*}
f_{\vec{n}}^{s+1}&=&\frac{\hat{M}(f_{\vec{n}}^{s})}{\left(1-\hat{M}(f_{\vec{n}}^{s})^\alpha  \right)^{1/\alpha}}\\
&\ge &\frac{\hat{M}(h_{\vec{n}}^{s})}{\left(1-\hat{M}(h_{\vec{n}}^{s})^\alpha  \right)^{1/\alpha}}\\
&\ge &\hat{M}(h_{\vec{n}}^{s})=h_{\vec{n}}^{s+1}.
\end{eqnarray*}
Here we used the inequality
\[
0\le x\le y < 1 \; \Longrightarrow \; \frac{x}{\left(1-x^\alpha  \right)^{1/\alpha}} \le \frac{y}{\left(1-y^\alpha  \right)^{1/\alpha}}.
\]
Thus, by the induction hypothesis, \eqref{5:ineqhf} holds for any $\tau \in \mathbb{Z}_{{}_\ge 0}$.
\end{proof}

%
%
%
%
Now we prove the statement$(2)$ of theorem \ref{th2.1} .\\
{\em Proof of theorem \ref{th2.1}  (2).} \ 

Suppose that $f^\tau_{\vec{n}}$ is a global solution of \eqref{eq:2.1}.
From proposition \ref{prop6.1},
\begin{eqnarray*}
\| f^\tau \|_1 &\geq& \sum\limits^\tau_{s=1}\sum\limits_{\vec{n},\vec{n}^\prime}{U^{\tau-s}_{\vec{n}-{\vec{n}^\prime}}H(g^{s-1}_{\vec{n}^\prime})} \\
&=& \sum\limits^\tau_{s=1}\sum\limits_{\vec{n}}{H(g^{s-1}_{\vec{n}})}.
\end{eqnarray*}
Here we applied $\sum_{\vec{n}}{U^\tau_{\vec{n}}}=1$.
By an elementary computation we can easily show
\begin{equation}\label{5:eqg}
H(g) \geq \frac{1}{\alpha}g^{1+\alpha}     \qquad (0 \leq g < 1).
\end{equation}
Hence 
\begin{equation}\label{5:ineqfg}
\| f^\tau \|_1 \geq \sum\limits_{s=1}^\tau\sum\limits_{\vec{n}}{\frac{1}{\alpha}(g^{s-1}_{\vec{n}})^{1+\alpha}}.
\end{equation}

From \eqref{5:ineqhf},
\begin{equation*}
g^{s-1}_{\vec{n}} \geq \hat{M}(h^{s-1}_{\vec{n}}) = h^s_{\vec{n}}=\sum_{\vec{n}^\prime}f_{\vec{n}^\prime}^0U_{\vec{n}-\vec{n}^\prime}^s.
\end{equation*}
There is at least one $\vec{n}_0 \in \mathbb{Z}^d$ such that $f_{\vec{n}_0}^0 >0$, and we put $c:=f_{\vec{n}_0}^0>0$.
Then
\[
h^s_{\vec{n}} \ge  c  U^s_{\vec{n}-\vec{n}_0},
\]
and
\begin{eqnarray}
\| f^\tau \|_1 &\geq& \sum\limits^\tau_{s=1}\sum\limits_{\vec{n}}{\frac{1}{\alpha}(h^s_{\vec{n}})^{1+\alpha}} \nonumber\\
&\geq& c^{1+\alpha}\frac{1}{\alpha}\sum\limits_{s=1}^\tau\sum\limits_{\vec{n}}{(U^s_{\vec{n}})^{1+\alpha}}\label{eq:6.8}.
\end{eqnarray}
Moreover,  from proposition \ref{prop3.2}, we obtain the asymptotic behaviour
\begin{equation}
U^\tau_{\vec{n}} \sim \displaystyle{2\left(\frac{d}{4\pi}\right)^{d/2}\tau^{-d/2}\exp{\left(-\frac{|\vec{n}|^2}{4\tau}\right)}} \ (\tau\to+\infty)\label{eq:6.9}.
\end{equation}
Let  $\delta$ and $\xi$ be  $\displaystyle \delta:=\frac{1}{\tau}$ and $ \xi:=\sqrt{2d\delta}$ respectively.
From $(\ref{eq:6.9})$, we obtain the following evaluation for $\alpha=2/d$ and $\tau \to +\infty \; (\delta \to +0, \ \xi \to +0)$. 
\begin{eqnarray*}
\displaystyle{\sum\limits_{\vec{n}}{(U^\tau_{\vec{n}})^{1+2/d}}} &\sim& 
\displaystyle{\sum\limits_{\vec{n}}{2^{1+2/d}\left(\frac{2}{4\pi\tau}\right)^{1+d/2}\exp{\left\{-\frac{|\vec{n}|^2}{4\tau}\left(1+\frac{2}{d}\right)\right\}}}} \ (\tau\to+\infty) \\
&=& \displaystyle{2^{1+2/d}\frac{d}{4\pi\tau}\frac{\xi^d}{(8\pi\tau\delta)^{d/2}}\sum\limits_{\vec{n}}{\exp{\left\{-\frac{(\xi|\vec{n}|)^2}{8\tau\delta}\left(\frac{1}{d}+\frac{2}{d^2}\right)\right\}}}} \\
&\sim& \displaystyle{\frac{d2^{2/d}}{2\pi\tau}\frac{1}{(8\pi )^{d/2}}\int_{\mathbb{R}^d}{\exp{\left\{-\frac{|\vec{x}|^2}{8}\left(\frac{1}{d}+\frac{2}{d^2}\right)\right\}}d\vec{x}}}\\
&=&\displaystyle{\frac{d2^{2/d}}{2\pi}\left(\frac{1}{d}+\frac{2}{d^2}\right)^{-d/2}\frac{1}{\tau}}.
\end{eqnarray*}
Since $\sum\limits_{s=1}^{\infty}{1/s}$ diverges,  $\sum\limits_{s=1}^\tau\sum\limits_{\vec{n}}{(U^s_{\vec{n}})^{1+2/d}} $ can take an arbitrarily large value.
From $(\ref{eq:6.8})$,  $ \| f^\tau \|_1$ also becomes arbitrarily large, which contradicts with proposition \ref{prop6.2}. 
Therefore, when $\alpha=2/d$, there exists no global solution and we have completed  the proof of theorem \ref{th2.1}.
$\qquad\endproof$
%
%
%
%

\vspace*{2mm}
{\bf Acknowledgements}
The authors would like to thank Profs. Atsushi Nagai, Ralph Willox and Dr. Mikio Murata for useful comments and discussions.


\end{document}